\documentclass[a4paper,11pt]{amsart}

\title{A lower bound for the number of Egyptian fractions}

\author{Sandro Bettin}
\address{Dipartimento di Matematica, Universit\`{a} di Genova, Via Dodecaneso 35, 16146 Genova, Italy}
\email{sandro.bettin@unige.it}

\author{Lo\"{\i}c Greni\'{e}}
\address{Dipartimento di Ingegneria Gestionale, dell'Informazione e della Produzione, Universit\`{a} degli
Studi di Bergamo, viale Marconi 5, 24044 Dalmine, Italy}
\email{loic.grenie@gmail.com}

\author{Giuseppe Molteni}
\address{Dipartimento di Matematica, Universit\`{a} di Milano, Via Saldini 50, 20133 Milano, Italy}
\email{giuseppe.molteni1@unimi.it}

\author{Carlo Sanna}
\address{Dipartimento di Scienze Matematiche, Politecnico di Torino, Corso Duca degli Abruzzi 24, 10129 Torino, Italy}
\email{carlo.sanna@polito.it}

\usepackage{amsmath}
\usepackage{amssymb}
\usepackage{amsthm}
\usepackage{geometry}
\usepackage{graphicx}
\usepackage{xcolor}
\usepackage{array}
\usepackage[colorlinks=true]{hyperref}
\usepackage{booktabs}
\usepackage{pgfplots}
\pgfplotsset{compat=1.18}

\newtheorem*{theorem*}{Theorem}
\newtheorem{theorem}{Theorem}
\newtheorem{lemma}{Lemma}
\theoremstyle{remark}

\newcommand{\Z}{\mathbb{Z}}
\newcommand{\R}{\mathbb{R}}

\newcommand{\EE}{{E}}
\newcommand{\EF}{{E}}
\newcommand{\UU}{\mathcal{U}}
\newcommand{\af}{\mathfrak{a}}
\newcommand{\eps}{\varepsilon}

\newcommand{\dd}{\mathrm{d}}
\newcommand{\card}[1]{{|{#1}|}}
\DeclareMathOperator{\lcm}{lcm}
\DeclareMathOperator{\Li}{Li}

\newcolumntype{L}{>{$}l<{$}}
\newcolumntype{R}{>{$}r<{$}}
\newcolumntype{C}{>{$}c<{$}}

\newenvironment{List}{\begin{list}{$\bullet$}{
\setlength{\labelwidth}{.3cm}
\setlength{\leftmargin}{.5cm}
}
}{\end{list}}

\allowdisplaybreaks[4]

\begin{document}

\begin{abstract}
An \emph{Egyptian fraction} is a sum of the form $1/n_1 + \cdots + 1/n_r$ where $n_1, \dots, n_k$
are distinct positive integers.
We prove explicit lower bounds for the cardinality of the set $\EE_N$ of rational numbers that can be
represented by Egyptian fractions with denominators not exceeding $N$.
More precisely, we show that for every integer $k \geq 4$ such that $\ln_k N \geq 3/2$ it holds
\[
\frac{\ln(\card{\EE_N})}{\ln 2}
\geq \Big(2 - \frac{3}{\ln_k N}\Big)\frac{N}{\ln N}\prod_{j=3}^{k} \ln_j N ,
\]
where $\ln_k$ denotes the $k$-th iterate of the natural logarithm. This improves on a previous result
of Bleicher and Erd\H{o}s who established a similar bound but under the more stringent condition
$\ln_k N\geq k$ and with a leading constant of $1$.\\
Furthermore, we provide some methods to compute the exact values of $|\EE_N|$ for large positive
integers $N$, and we give a table of $|\EE_N|$ for $N$ up to $154$.
\end{abstract}

\maketitle

\section{Introduction}

An \emph{Egyptian fraction} is a sum of the form $1/n_1 + \cdots + 1/n_r$ where $n_1, \dots, n_k$
are pairwise distinct positive integers.
It is well known that every positive rational number can be written as an Egyptian fraction.
In general, this representation is not unique.
Several authors studied the properties of Egyptian fractions.
For example, Yokota~\cite{Yokota,Yokota2} (see also~\cite{TenenbaumYokota}) proved that, for all positive
integers $a < b$, the rational number $a/b$ is represented by an Egyptian fraction with denominators of
size $O\big(b (\ln b)^{1+\eps}\big)$; while Vose~\cite{Vose} showed that $a/b$ can be represented using
at most $O\big(\!\sqrt{\ln b}\big)$ denominators.
In the opposite direction of these ``efficient'' representations, Martin~\cite{Martin,Martin2} showed
that every positive rational number can be represented using a very ``dense'' set of denominators.
Furthermore, Croot~\cite{Croot} proved that all positive integers less than $\big[\sum_{n\leq N}\frac1n\big]$
can be represented using denominators up to $N$. We refer to the paper by Bloom and
Elsholtz~\cite{BloomElsholtz} for a recent survey on the topic of Egyptian fractions.

In this paper we are interested in the cardinality of the set
\begin{equation}\label{equ:definition-of-E_N}
    \EE_N := \bigg\{\sum_{n = 1}^N \frac{t_n}{n}\colon t_1, \dots, t_N \in \{0,1\} \bigg\},\qquad N\in\Z_{\geq 1},
\end{equation}
of Egyptian fractions that employ denominators up to $N$. It is clear from the definition that $|\EE_N|$
is increasing, and in fact it is strictly increasing, since $\EE_N\setminus \EE_{N-1}$ contains at least
the $N$-th harmonic number.

Bleicher and Erd\H{o}s~\cite[Th.~2 and~3]{BleicherErdos2} proved that
\begin{equation}\label{eq:1C}
\alpha\frac{N}{\ln N}\prod_{j=3}^k\ln_{j} N \leq \ln(\card{\EE_N}) \leq \frac{N\ln_{k}N }{\ln N}\prod_{j=3}^k\ln_{j} N
\end{equation}
where $\alpha=e^{-1}$, $k$ is any positive integer such that $\ln_{2k} N\geq 1$, where $\ln_k$ denotes
the $k$-th iteration of the natural logarithm. In~\cite[Cor.~1, 2, and 3]{BleicherErdos3} they improved
the lower bound, increasing the value of the admissible $\alpha$ to $\alpha=\ln 2$ and relaxing the
condition on $k$ to $\ln_{k} N\geq k$. The problem of determining the size of $\card{\EE_N}$ is then also
mentioned in a problem list by Erd\H{o}s and Graham~\cite[p.~43]{ErdosGraham}.

In this paper, we give a new improvement to the lower bound in~\eqref{eq:1C} showing that the condition
on $k$ can be further relaxed to $\ln_{k} N\geq 3/2$, essentially the weakest possible for a bound of
this shape, while $\alpha$ can be taken to be $2\ln 2-\eps$ for any fixed $\eps>0$ provided that $\ln_k N$
is large enough.  More specifically, we prove the following result.
\begin{theorem}\label{thm:lower-bound}
For all positive integers $k$ and $N$, it holds
\[
\ln(\card{\EE_N})
\geq
2\ln 2\,\frac{N}{\ln N}
\times
\begin{cases}
\displaystyle
1        & \text{if } \ln_2 N \geq 1,\\[.3cm]
\displaystyle
\ln_3 N & \text{if } \ln_3 N \geq 1,\\
\displaystyle
\Big(1 - \frac{3/2}{\ln_k N}\Big)\prod_{j=3}^{k} \ln_j N & \text{if $k\geq 4$ and } \ln_k N \geq 3/2.
\end{cases}
\]
\end{theorem}

We note that the improvement over Bleicher and Erd\H{o}s due to the weaker condition in $k$ translates to
an improvement in the order of growth of the lower bound. In fact, picking the largest admissible value of
$k$ in both bounds, Theorem~\ref{thm:lower-bound} yields a bound that goes to infinity faster than the
bound proven in~\cite[Cor.~3]{BleicherErdos3} since for extremely large values for $N$, larger values of
$k$ are allowed.
Further refinements to the constants in Theorem~\ref{thm:lower-bound} are certainly possible, as our main
goal was to establish the qualitative improvement resulting from relaxing the condition on $k$. For
example, the leading factor $2$ could be slightly improved by taking full advantage of the constant $2.2$
appearing in~\eqref{eq:6C}.

As in Bleicher and Erd\H{o}s' paper, the proof of the theorem proceeds by giving a lower bound on the
number of $N$ such that $|E_N|=2|E_{N-1}|$ using a recursive approach. In particular, letting $\UU(x)$ be
the number of such occurrences with $N\leq x$, in Lemma~\ref{lem:6C} we prove $\card{\UU(x)} \geq
\frac{x}{\ln x}  \int_{1}^{y} \frac{\card{\UU(v)}}{v^2} \dd v$ if $x$ is large enough. Together with a
bound for $\UU(x)$ for small $x$, the recursive use of the afore-mentioned inequality leads to the
claimed bound.
\medskip

For large positive integers $N$, computing $|\EE_N|$ directly from its definition~\eqref{equ:definition-of-E_N}
is infeasible.
In fact, this requires the computation of $2^N$ sums of rational numbers with denominators of size up
to $\lcm\{1,2,\dots,N\}$.
In sequence A072207 of OEIS~\cite{OEIS} it is possible to find the values of $|E_N|$ for $N=1,\dots,83$.
In Section~\ref{sec:computation} we explain the methods that we used to compute $|E_N|$ up to $N=154$.
These values are provided in Table~\ref{tab:2} at the end of the paper.
The main difficulty to extend the computation to $N > 154$ lies in the large quantity of memory that is
necessary to store the intermediate results, not in execution time.

Analysing the values of the quotient $|\EE_N|/|\EE_{N{-}1}|$ for the numbers that we have at our
disposal, we see that most of the times $|\EE_N|/|\EE_{N{-}1}|$ is either equal to $2$ or very close to
$1$, as expected since the  upper-bound in~\eqref{eq:1C} implies that $|\EE_N|/|\EE_{N{-}1}|\leq1+\eps$
for almost all $N$ for any fixed $\eps>0$.
More precisely, let
\[
D(n)
:= \Big|\Big\{
        N\leq 154\colon \frac{1}{\ln 2}\ln\Big(\frac{|\EE_N|}{|\EE_{N{-}1}|}\Big) \in \Big(\frac{n{-1}}{10},\frac{n}{10}\Big]
        \Big\}
   \Big|,
\qquad n=1,\ldots,10.
\]
Table~\ref{tab:1} shows that $D(n)$ is concentrated around $1$ and $10$, that is, $|\EE_N|/|\EE_{N{-}1}|$
is concentrated around $1$ and $2$.
Furthermore, a direct check reveals that all but one of the cases counted in the last column of
Table~\ref{tab:1} satisfy $|\EE_N|/|\EE_{N{-}1}| = 2$.
\begin{table}[ht]
\centering
\begin{tabular}{R|RRRRRRRRRR}
  \toprule
   n    &  1 & 2 & 3 & 4 & 5 & 6 & 7 & 8 & 9 & 10 \\
  \midrule
   D(n) & 49 & 3 & 2 & 4 & 1 & 2 & 2 & 2 & 3 & 85 \\
  \bottomrule
\end{tabular}
\medskip
\caption{The values of $D(n)$.}
\label{tab:1}
\end{table}

As an additional point in favor of our approach, we show the graph of
$|\UU(N)|\ln 2/\ln|\EE_N|$, for $1\leq N\leq 154$. As one can see, the sequence
does not appear to tend to $0$, much more to something above $0.85$.
Therefore, the doublings seem to explain most of the increase of $|\EE_N|$.

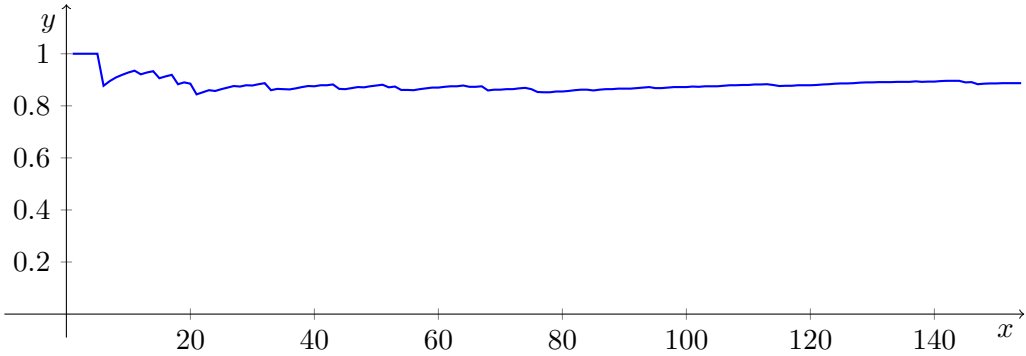
\begin{center}
\begin{tikzpicture}
    \begin{axis}[
            xmin=-10,xmax=154.49,
            ymin=-0.09,ymax=1.19,
            width=15cm,
            height=6cm,
            axis x line=middle,
            axis y line=middle,
            axis line style=->,
            xlabel={$x$},
            xlabel style={anchor=north east},
            ylabel={$y$},
            ylabel style={anchor=north east},
            ytick distance=0.2,
            ]
        \draw[blue,thick]
                (  1,1) -- (  2,1) -- (  3,1) -- (  4,1)
             -- (  5,1) -- (  6,0.877) -- (  7,0.895) -- (  8,0.909)
             -- (  9,0.919) -- ( 10,0.928) -- ( 11,0.935) -- ( 12,0.921)
             -- ( 13,0.928) -- ( 14,0.933) -- ( 15,0.906) -- ( 16,0.913)
             -- ( 17,0.919) -- ( 18,0.883) -- ( 19,0.890) -- ( 20,0.885)
             -- ( 21,0.844) -- ( 22,0.852) -- ( 23,0.860) -- ( 24,0.857)
             -- ( 25,0.864) -- ( 26,0.870) -- ( 27,0.876) -- ( 28,0.874)
             -- ( 29,0.879) -- ( 30,0.878) -- ( 31,0.883) -- ( 32,0.887)
             -- ( 33,0.860) -- ( 34,0.865) -- ( 35,0.864) -- ( 36,0.863)
             -- ( 37,0.867) -- ( 38,0.872) -- ( 39,0.876) -- ( 40,0.875)
             -- ( 41,0.879) -- ( 42,0.879) -- ( 43,0.882) -- ( 44,0.865)
             -- ( 45,0.864) -- ( 46,0.868) -- ( 47,0.872) -- ( 48,0.871)
             -- ( 49,0.875) -- ( 50,0.878) -- ( 51,0.881) -- ( 52,0.871)
             -- ( 53,0.874) -- ( 54,0.861) -- ( 55,0.861) -- ( 56,0.860)
             -- ( 57,0.864) -- ( 58,0.867) -- ( 59,0.870) -- ( 60,0.870)
             -- ( 61,0.873) -- ( 62,0.875) -- ( 63,0.875) -- ( 64,0.878)
             -- ( 65,0.873) -- ( 66,0.873) -- ( 67,0.875) -- ( 68,0.859)
             -- ( 69,0.862) -- ( 70,0.862) -- ( 71,0.864) -- ( 72,0.864)
             -- ( 73,0.867) -- ( 74,0.869) -- ( 75,0.864) -- ( 76,0.853)
             -- ( 77,0.852) -- ( 78,0.852) -- ( 79,0.855) -- ( 80,0.855)
             -- ( 81,0.857) -- ( 82,0.860) -- ( 83,0.862) -- ( 84,0.862)
             -- ( 85,0.859) -- ( 86,0.862) -- ( 87,0.864) -- ( 88,0.864)
             -- ( 89,0.866) -- ( 90,0.866) -- ( 91,0.866) -- ( 92,0.868)
             -- ( 93,0.870) -- ( 94,0.872) -- ( 95,0.868) -- ( 96,0.868)
             -- ( 97,0.870) -- ( 98,0.872) -- ( 99,0.872) -- (100,0.872)
             -- (101,0.874) -- (102,0.873) -- (103,0.875) -- (104,0.875)
             -- (105,0.875) -- (106,0.877) -- (107,0.879) -- (108,0.879)
             -- (109,0.880) -- (110,0.880) -- (111,0.882) -- (112,0.882)
             -- (113,0.883) -- (114,0.880) -- (115,0.876) -- (116,0.877)
             -- (117,0.877) -- (118,0.879) -- (119,0.879) -- (120,0.879)
             -- (121,0.880) -- (122,0.882) -- (123,0.883) -- (124,0.885)
             -- (125,0.886) -- (126,0.886) -- (127,0.887) -- (128,0.889)
             -- (129,0.890) -- (130,0.890) -- (131,0.891) -- (132,0.891)
             -- (133,0.891) -- (134,0.892) -- (135,0.892) -- (136,0.892)
             -- (137,0.894) -- (138,0.892) -- (139,0.893) -- (140,0.893)
             -- (141,0.895) -- (142,0.896) -- (143,0.896) -- (144,0.896)
             -- (145,0.890) -- (146,0.891) -- (147,0.883) -- (148,0.885)
             -- (149,0.886) -- (150,0.886) -- (151,0.887) -- (152,0.887)
             -- (153,0.887) -- (154,0.887)
        ;
    \end{axis}
\end{tikzpicture}
\end{center}

\section*{Acknowledgements}
S.~Bettin is member of the INdAM group GNAMPA.
L.~Greni\'{e}, G.~Molteni and C.~Sanna are members of the INdAM group GNSAGA.
C.~Sanna is a member of CrypTO, the group of Cryptography and Number Theory of the Politecnico di Torino.
S.~Bettin is partially supported by PRIN 2022 “The arithmetic of motives and L-functions”, by the
Curiosity Driven grant “Value distribution of quantum modular forms” of the University of Genoa,
funded by the European Union--NextGenerationEU, and by the MIUR Excellence Department Project
awarded to Dipartimento di Matematica, Università di Genova, CUP D33C23001110001.

The authors thank C. Elsholtz for pointing them to the works of Bleicher and Erd\H{o}s mentioned in this
paper.
The computations of $\EE_N$ have been performed on the PlaFRIM platform in Bordeaux: the authors warmly
thank Karim~Belabas for his operative assistance.

\section*{Notation}
The letter $p$ is reserved for prime numbers.
For every finite set $S$, let $\card{S}$ be the cardinality of $S$.
For every real number $x$, let $\pi(x):=|\{p\leq x\}|$ be the prime-counting function.
For every integer $k \geq 0$, let $e_k$ be the tower of $k$ exponentials in $e$, that is $e_0:=1$ and
$e_{k}= e^{e_{k-1}}$ for $k\geq 1$.
For every integer $k \geq 0$, let $\ln_k$ the $k$-th iteration of the natural logarithm.
Note that $\ln_k e_k =1$ and $\ln_k e_{k-1}=0$ for every positive integer $k$.

\section{Proof of Theorem~\ref{thm:lower-bound}}\label{sec:1C}

Define the following set of positive integers
\[
\UU := \Big\{N \geq 1\colon \sum_{n=1}^{N-1} \frac{w_n}{n}
\neq \frac{1}{N} \text{ for all } w_1, \dots, w_{N-1} \in \{-1,0,+1\}\Big\} .
\]
Note that the set $\UU$ is stable by divisors, i.e., if $N\in\UU$ then every divisor of $N$ is in $\UU$
as well. For every real number $x$, let $\UU(x) := \UU \cap [1, x]$.
\begin{lemma}\label{lem:0C}
    Let $N$ be a positive integer.
    Then $N \in \UU$ if and only if $\card{\EE_N} = 2\card{\EE_{N - 1}}$.
\end{lemma}
\begin{proof}
The definition of $\EE_N$ yields at once $\EE_N = \EE_{N - 1} \cup (\EE_{N - 1} + 1/N)$ (setting $\EE_0
:= \{0\}$). By the definition of $\UU$, the union is disjoint if and only if $N\in\UU$.
Since $\EE_{N - 1}$ and $\EE_{N - 1} + 1/N$ have the same cardinality, their union is disjoint if and
only if $\card{\EE_N} = 2\card{\EE_{N - 1}}$.

\end{proof}
\begin{lemma}\label{lem:1C}
    $\card{\EE_N} \geq 2^{\card{\UU(N)}}$ for every positive integer $N$.
\end{lemma}
\begin{proof}
    The claim follows immediately from Lemma~\ref{lem:0C}.
\end{proof}

\noindent %
Thanks to Lemma~\ref{lem:1C}, to produce a lower bound for $\card{\EE_N}$, it suffices to give a lower bound
for $\card{\UU(N)}$.

For each positive integer $m$, let
\[
d_m := \lcm\{1, \dots, m\},\qquad g_m := d_m \sum_{j = 1}^m \frac 1  j.
\]
We will require the following bound for $\pi(g_m)$.
\begin{lemma}\label{lem:3C}
For any positive integer $m$, we have $\pi(g_m)\leq \frac{18\cdot 3^m}{m\ln(18\cdot 3^m)}$.
\end{lemma}
\begin{proof}
From~\cite[p.~228]{Rosser3} we know that $d_m = \exp(\psi(m)) \leq \exp(1.04 m)$. On the other hand,
\[
\sum_{j=1}^m \frac{1}{j} \leq 1 + \int_1^m \frac{\mathrm{d}t}{t} = \ln(em) .
\]
Hence, $g_m \leq \exp(1.04 m) \ln(em)$. This implies, by a quick computation, that $g_m<\frac{15}{m}\cdot 3^m$
for all integers $m\geq 46$, and a direct computation shows that this bound also holds for $1\leq m\leq 45$.
Then, writing $u:=15\cdot 3^m/m$ by~\cite[Th.~1]{RosserSchoenfeld}  we have
\[
\pi(g_m)\leq \pi(u)
 \leq \frac{u}{\ln u}\Big(1+\frac{3}{2\ln u}\Big)
 \leq \frac{18\cdot 3^m}{m\ln(18\cdot 3^m)}
\]
for $m\geq 24$. The lemma then follows by direct verification of the cases $1\leq m\leq 23$.

\end{proof}

We say that a prime $p$ is \emph{compatible} with $m$ if $p$ does not divide the numerator of any rational
of the form $1/m-\sum_{j = 1}^{m-1} w_j / j$ with $w_j\in \{ -1, 0, +1 \}$.
Thus, if $m\notin\UU$ then $1/m-\sum_{j = 1}^{m-1} w_j / j=0$ for some choice of $w_j$ and no prime is
compatible with $m$. On the other hand, if $m\in\UU$ and $p > g_m$ then $p$ is compatible with $m$.
The following result shows how to generate many elements in $\UU$ in a recursive way, starting with the
evident fact that $1\in\UU$.
\begin{lemma}\label{lem:2C}
Let $m$ be in $\UU$ and $k$ be any positive integer.
Then $mp^k\in \UU$, for any prime $p$ compatible with $m$. In particular, this holds for any prime $p>g_m$.
\end{lemma}
\begin{proof}
Suppose by contradiction that there exist $m$, $p$ and $k$ as above such that $N := mp^k \notin \UU$. %
Hence, by definition, there exist $w_1, \dots, w_{N - 1}$ in $\{-1,0,+1\}$ such that
\[
\sum_{n=1}^{N-1} \frac{w_n}{n} = \frac{1}{N} = \frac{1}{mp^k}.
\]
We split the sum according to whether $p^k\mid n$ or not and multiplying by $p^kD$, where
$D$ is the denominator of $\frac{1}{m}-\sum_{j=1}^{m-1} \frac{w_{p^kj}}{j}$. We obtain
\[
p^kD\sum_{\substack{n=1 \\ p^k\nmid\, n}}^{N-1} \frac{w_n}{n}
= D\Bigl(\frac{1}{m} - \sum_{j = 1}^{m - 1} \frac{w_{p^kj}}{j}\Bigr).
\]
The right-hand side is the numerator of $\frac{1}{m}-\sum_{j=1}^{m-1} \frac{w_{p^kj}}{j}$ and is divisible
by $p$, since the left-hand side is. It follows that $p$ is not compatible with $m$, which is a contradiction.
\end{proof}
Notice that, since $g_1=1$, the case $m=1$ in Lemma~\ref{lem:2C} already proves that all prime powers
are elements of $\UU$.
\begin{lemma}\label{lem:4C}
We have
\begin{multline}\label{equ:U100}
\UU(100) = \{
1, 2, 3, 4, 5, 7, 8, 9, 10, 11, 13, 14, 16, 17, 19, 22, 23, 25, 26, 27, 29,\\
31, 32, 34, 37, 38, 39, 41, 43, 46, 47, 49, 50, 51, 53, 57, 58, 59, 61, 62,\\
64, 67, 69, 71, 73, 74, 79, 81, 82, 83, 86, 87, 89, 92, 93, 94, 97, 98
\}.
\end{multline}
\end{lemma}

\begin{proof}
First, we list the numbers that belong to $\UU(100)$ thanks to Lemma~\ref{lem:2C}.
\begin{List}
\item We know that $1$ is in $\UU$.
\item $g_1= 1$, hence if $2\leq p\leq 100$, $k\geq 1$ and $N=p^k$, then $N\in \UU$.
\item $g_2= 3$, hence if $5\leq p\leq 47$, $k\geq 1$ and $N=2p^k$, then $N\in \UU$.
\item $g_3= 11$, hence if $13\leq p\leq 31$ and $N=3p$, then $N\in \UU$.
\item $23$ is compatible with $4$, hence $92=4\cdot 23$ is in $\UU$.
\end{List}
This proves that $\UU(100)$ contains all the numbers listed in~\eqref{equ:U100}. We now show that
no other number belongs to $\UU(100)$.

Every multiple of $6$, $15$, $20$, $21$, $28$, $33$, $35$, or $44$ does not belong to $\UU$ because
\begin{center}
\begin{tabular}{lll}
    $\displaystyle\frac{1}{6}  = \frac{1}{2} -\frac{1}{3}$,  &
    $\displaystyle\frac{1}{15} = \frac{1}{6} -\frac{1}{10}$, &
    $\displaystyle\frac{1}{20} = \frac{1}{4} -\frac{1}{5}$, \\[2.5ex]
    \hspace{-0.5em}
    $\displaystyle\frac{1}{21} = \frac{1}{7} +\frac{1}{14}-\frac{1}{6}$,  &
    $\displaystyle\frac{1}{33} = \frac{1}{6} -\frac{1}{11}-\frac{1}{22}$, &
    $\displaystyle\frac{1}{44} = \frac{1}{33}+\frac{1}{12}-\frac{1}{11}$.
\end{tabular}
\end{center}
Furthermore, $52$, $55$, $65$, $68$, $76$, $85$, $91$, and $95$ do not belong to $\UU$ because
\begin{center}
    \begin{tabular}{llll}
        $\displaystyle\frac{1}{52}= \frac{1}{12}-\frac{1}{26}-\frac{1}{39}$, &
        $\displaystyle\frac{1}{55}=-\frac{1}{20}+\frac{1}{22}+\frac{1}{44}$, &
        $\displaystyle\frac{1}{65}=-\frac{1}{10}+\frac{1}{13}+\frac{1}{26}$, \\[2.5ex]
        $\displaystyle\frac{1}{68}= \frac{1}{4} -\frac{1}{6} -\frac{1}{17}-\frac{1}{34}+\frac{1}{51}$, &
        $\displaystyle\frac{1}{76}= \frac{1}{12}-\frac{1}{19}-\frac{1}{57}$, &
        $\displaystyle\frac{1}{85}= \frac{1}{10}-\frac{1}{17}-\frac{1}{34}$, \\[2.5ex]
        $\displaystyle\frac{1}{91}= \frac{1}{42}-\frac{1}{78}$,              &
        $\displaystyle\frac{1}{95}= \frac{1}{20}-\frac{1}{38}-\frac{1}{76}$.
    \end{tabular}
\end{center}
The claim is proved.

\end{proof}
\begin{lemma}\label{lem:5C}
Let $x$ be a real number.
Then
\begin{align}
\card{\UU(x)}&\geq \phantom{\frac{11}{11}} 2\,\frac{x}{\ln x}
\qquad\text{ if } x\geq 13,                            \label{eq:2C}\\
\card{\UU(x)}&\geq \frac{137}{60}           \,\frac{x}{\ln x}
\qquad\text{ if } x\geq 1000.                          \label{eq:3C}
\end{align}
\end{lemma}
\begin{proof}
The explicit description of $\UU(100)$ in Lemma~\ref{lem:4C} implies~\eqref{eq:2C} for $x\leq 100$.
Hence, assume that $x > 100$. Since $1$, $2$, $3$ and $4$ belong to $\UU$, we find that all elements of
the form $p$, $2p$ (with $p>g_3=3$), $3p$ (with $p>g_3=11$) and $4p$ (with $p>g_4=25$) belong to $\UU$.
In $\UU(x)$ we thus have $\pi(x)$ elements of the first type, $\pi(x/2)-\pi(3)$ elements of the second,
$\pi(x/3)-\pi(11)$ elements of the third and $\pi(x/4)-\pi(25)$ elements of the fourth, i.e.,
$\pi(x) + \pi(x/2) + \pi(x/3) + \pi(x/4) - 16$ elements of one of the above mentioned forms.
Moreover, $1$, $4$, $8$, $16$, $32$, $64$, $9$, $27$, $81$, $92$, $25$, $49$, $50$ and $98$ belong to
$\UU$ and are not of the aforementioned form. Hence for $x\geq 100$, we have
\[
\card{\UU(x)}
\geq \pi(x) + \pi\Big(\frac{x}{2}\Big) + \pi\Big(\frac{x}{3}\Big) + \pi\Big(\frac{x}{4}\Big) - 2.
\]
In~\cite[Th.~2, Cor.~1]{RosserSchoenfeld} it is proved that $\pi(t)\geq t/\ln t$ when $t\geq 17$. Thus,
\[
\card{\UU(x)}
\geq \frac{x}{\ln x} + \frac{x/2}{\ln(x/2)} + \frac{x/3}{\ln(x/3)} + \frac{x/4}{\ln(x/4)} - 2
\geq 2\,\frac{x}{\ln x},
\]
where the last inequality is elementary and proves the claim.

Assume now $x\geq 1000$. Since $5\in\UU$, we also have elements of the form $5p$ (with $p>g_5=137$) in
$\UU$, and there are $\pi(x/5)-\pi(137)=\pi(x/5)-33$ elements of this form.

Adding these terms to the previous lower bound, we find that
\[
\card{\UU(x)}
\geq \frac{x}{\ln x} + \frac{x/2}{\ln(x/2)} + \frac{x/3}{\ln(x/3)} + \frac{x/4}{\ln(x/4)} + \frac{x/5}{\ln(x/5)} - 35
\geq \frac{137}{60}\,\frac{x}{\ln x},
\]
where the last inequality is once again elementary and proves~\eqref{eq:3C}.

\end{proof}
\begin{lemma}\label{lem:6C}
Assume $y\geq 1$ and $x\geq 18\cdot 3^y$. Then
\begin{equation}\label{eq:4C}
\card{\UU(x)} \geq \frac{x}{\ln x}  \int_{1}^{y} \frac{\card{\UU(v)}}{v^2} \dd v.
\end{equation}
\end{lemma}
\begin{proof}
By Lemmas~\ref{lem:2C} and~\ref{lem:3C} we have
\begin{align*}
\card{\UU(x)}
&\geq \sum_{m\in \UU(y)} \Big(\pi\!\Big(\frac{x}{m}\Big) - \pi(g_m)\Big)
 \geq \sum_{m\in \UU(y)} \Big(\pi\!\Big(\frac{x}{m}\Big) - \frac{18\cdot 3^m}{m\ln(18\cdot 3^m)}\Big).
\end{align*}
Since $x/m\geq x/y\geq 18\cdot 3^y/y\geq 17$ for every $m\leq y$, we have $\pi(x/m)\geq
\frac{x/m}{\ln(x/m)}\geq \frac{x/m}{\ln x}$ (by~\cite[Th.~2, Cor.~1]{RosserSchoenfeld}), hence
\begin{align*}
\card{\UU(x)}&\geq \sum_{m\in \UU(y)} \Big(\frac{x}{m\ln x} - \frac{18\cdot 3^m}{m\ln(18\cdot 3^m)}\Big).
\end{align*}
We write the right-hand side as a Stieltjes integral, so that by partial summation it becomes
\begin{align*}
&\int_{1^-}^{y^+} \Big(\frac{x}{v\ln x} - \frac{18\cdot 3^v}{v\ln(18\cdot 3^v)}\Big) \dd \card{\UU(v)} \\
&\hspace{5em}= \Big(\frac{x}{v\ln x} - \frac{18\cdot 3^v}{v\ln(18\cdot 3^v)}\Big)\card{\UU(v)}\Big|_{1^-}^{y^+}
  - \int_{1}^{y} \card{\UU(v)}
      \Big(
       - \frac{x}{v^2\ln x}
       - \frac{\dd}{\dd v}\Big(\frac{18\cdot 3^v}{v\ln(18\cdot 3^v)}\Big)
      \Big) \dd v\\
&\hspace{5em} = \Big(\frac{x}{\ln x} - \frac{18\cdot 3^y}{\ln(18\cdot 3^y)}\Big)\frac{\card{\UU(y)}}{y}
  + \frac{x}{\ln x}\int_{1}^{y} \frac{\card{\UU(v)}}{v^2}\dd v
  + \int_{1}^{y}
      \card{\UU(v)} \frac{\dd}{\dd v}\Big(\frac{18\cdot 3^v}{v\ln(18\cdot 3^v)}\Big) \dd v.
\end{align*}
The first term is positive, since $e<18\cdot 3^y\leq x$. Also the derivative in the last integral is
positive for $v\geq 2$.
Thus, assuming that $y\geq 2$ we get
\begin{align*}
\card{\UU(x)}
&\geq \frac{x}{\ln x}\int_{1}^{y} \frac{\card{\UU(v)}}{v^2}\dd v
  + \int_{1}^{2}
      \card{\UU(v)}\frac{\dd}{\dd v}\Big(\frac{18\cdot 3^v}{v\ln(18\cdot 3^v)}\Big) \dd v.
\intertext{Since $\card{\UU(y)} = 1$ for $y\in [1,2)$, this is}
&=   \frac{x}{\ln x}\int_{1}^{y} \frac{\card{\UU(v)}}{v^2}\dd v
  + \frac{18\cdot 3^v}{v\ln(18\cdot 3^v)}\Big|_{1}^{2}
\geq \frac{x}{\ln x}\int_{1}^{y} \frac{\card{\UU(v)}}{v^2}\dd v,
\end{align*}
which concludes the proof in this case. On the other hand, when $y\in[1,2)$ the claim states
\[
\card{\UU(x)}
\geq \frac{x}{\ln x}\int_{1}^{y}\frac{\card{\UU(v)}}{v^2}\dd v
=    \frac{x}{\ln x}\Big(1-\frac{1}{y}\Big),
\]
which is true by Lemma~\ref{lem:5C}.
\end{proof}
We want to simplify the integral recursion for $\card{\UU(x)}$ described in Lemma~\ref{lem:6C}. For this
purpose, let $G$ be defined by
\begin{equation}\label{eq:5C}
G(z) := \frac{\ln x}{x}\card{\UU(x)},
\quad\text{where }
x := \exp(\exp (z)).
\end{equation}
We set $y:=e^{z-1/4}$, so that the condition $x=e^{(1-e^{-1/4}\ln 3)e^z}\cdot 3^y\geq 18\cdot 3^y$ and
$y\geq 1$ are verified for $z\geq 3$.

Since $\ln\ln y=\ln(z-1/4)$, we can rewrite~\eqref{eq:4C} as
\[
G(z) \geq \int_{-\infty}^{\ln(z-1/4)} \hspace{-0.5cm} G(w) \dd w\quad \forall z\geq3.
\]

The part of the integral with $w\leq 1$ contributes $\int_{-\infty}^{1} G(w)\dd w {=}
\int_{1}^{\exp(\exp(1))} \frac{\card{\UU(v)}}{v^2}\dd v
\geq 2.2$ by Lemma~\ref{lem:4C} and a direct computation. Hence,
\begin{equation}\label{eq:6C}
G(z)
\geq 2.2 + \int_{1}^{\ln(z - 1/4)} \hspace{-0.5cm} G(w) \dd w
\quad
\forall z\geq 3.
\end{equation}
For any $k\geq 1$, let $h_k$ be defined recursively as $h_1:=1$ and
$h_k:=\exp(h_{k-1})+1/4$ for $k\geq 2$, and let $T_k\colon [0, +\infty)\to \R$ be the function defined
by the identities
\[
\begin{aligned}
T_{1}(z) &:= \ln z                                                         && \text{if } z > h_1, \\
T_{k}(z) &:= \int_{h_{k-1}}^{\ln(z - 1/4)} \hspace{-0.5cm} T_{k-1}(w)\dd w && \text{if } z > h_k
\quad \forall k\geq 2,
\end{aligned}
\]
and $0$ everywhere else.
By induction on $k$, one verifies that $h_k \in [e_{k-1},e_k]$ and that each $T_k$ is a continuous and
increasing function which is zero in $[0,h_k]$ and positive in $(h_k,+\infty)$.
We note that the definition of $T_k$ with $k\geq 2$ is modeled on~\eqref{eq:6C}, where we removed the
additive constant $2.2$ that appears there, since preliminary calculations show that its presence leads to
various complications that are not compensated by any truly significant improvement.

The following lower bound essentially comes from~\eqref{eq:6C}.
\begin{lemma}\label{lem:7C}
Let $k\geq 1$. Then for all $z>0$ we have
\begin{equation}\label{eq:7C}
G(z) \geq 2\,T_k(z).
\end{equation}
\end{lemma}
\begin{proof}
We first prove the statement for $k=1$. When $z\leq 1$ the claim is clear since $T_1(z)$ is zero there.
By~\eqref{eq:2C} we know that $G(z) \geq 2$ holds for $z\geq 1$ since $\ln\ln 13<1$. If $z >3$ the recursive
lower bound~\eqref{eq:6C} then implies that $G(z)\geq 2.2 + \int_1^{\ln(z - 1/4)}2\,\dd u = 2\ln(z-1/4) +
0.2\geq 2T_1(z)$.
In the intermediate ranges $[1,e]$ and $[e,3]$ we use~\eqref{eq:2C} in the first case, so that $G(z)\geq
2 = 2 T_1(e) \geq 2 T_1(z)$, and~\eqref{eq:3C} in the second case, so that $G(z)\geq 137/60 \geq 2 T_1(3)
\geq 2 T_1(z)$ once again.

Assume $k = 2$. If $z \leq h_2 = e + 1/4 = 2.96\dots$, then $G(z)\geq 2T_2(z)$ is trivially satisfied, as
the right-hand side is $0$ there, and the same lower bound also holds when $z \in [h_2,3]$ since
by~\eqref{eq:2C} we have $G(z)\geq 2 \geq 2 T_2(3) \geq 2 T_2(z)$.
When $z\geq 3$ we use~\eqref{eq:6C} and~\eqref{eq:7C}, in the case $k=1$, obtaining
\[
G(z)
\geq 2.2 + 2\int_{1}^{\ln(z - 1/4)} \hspace{-0.5cm} T_1(w) \dd w
\geq 2T_2(z).
\]

Finally, assume $k\geq 3$. Once again, the bound is trivially true when $z \leq h_k$. Let $z\geq h_k$.
Then $z\geq 3$ so that~\eqref{eq:6C} and the inductive hypothesis stating that $G(z)\geq 2T_{k-1}(z)$ give
\[
G(z)
\geq 2.2 + 2\int_{1}^{\ln(z - 1/4)} \hspace{-0.5cm} T_{k-1}(w) \dd w.
\]
The function inside the integral is zero when $w< h_{k-1}$, therefore we can replace the lower integral
with $h_{k-1}$ and the claim follows by discarding the additive constant $2.2$.
\end{proof}
\begin{lemma}\label{lem:8C}
Suppose that for a certain integer $k\geq 3$ there exists a positive constant $a_{k-1}\geq 1$ such that
\begin{equation}\label{eq:8C}
T_{k-1}(z) \geq \prod_{j=1}^{k-1} \ln_j z - a_{k-1}\prod_{j=1}^{k-2} \ln_j z
\end{equation}
when $z\geq e_{k-1}$. %
Then, we have
\begin{equation}\label{eq:9C}
T_k(z) \geq \prod_{j=1}^k \ln_j z - a_k\prod_{j=1}^{k-1} \ln_j z
\end{equation}
when $z \geq e_{k}$, with
\[
a_k := a_{k-1} + \frac{1}{e_{k-2}} + \frac{k-2}{e_{k-2}e_{k-3}}.
\]
\end{lemma}
\begin{proof}
We start by noticing that the claim in~\eqref{eq:8C} holds trivially in the range $z \in
(e_{k-2},e_{k-1})$. In fact, the assumption $a_{k-1}\geq 1$ implies that the right-hand side
of~\eqref{eq:8C} is negative whenever $z < e_{k-1}$, while $T_{k-1}$ is always nonnegative.

Let $z \geq e_k$ and recall that $h_k \in [e_{k-1},e_k)$, so that $z \geq h_k$. Then
\[
T_k(z)
= \int_{h_{k-1}}^{\ln(z-1/4)} \hspace{-0.5cm} T_{k-1}(u) \dd u
= \int_{e_{k-2}}^{\ln(z-1/4)} \hspace{-0.5cm} T_{k-1}(u) \dd u
\]
since $T_{k-1}$ is $0$ in $[e_{k-2},h_{k-1}]$. By the assumed lower bound for $T_{k-1}$ we have
\[
T_k(z)
\geq         \int_{e_{k-2}}^{\ln(z - 1/4)} \prod_{j=1}^{k-1}\ln_{j} w\, \dd w
     -a_{k-1}\int_{e_{k-2}}^{\ln(z - 1/4)} \prod_{j=1}^{k-2}\ln_{j} w\, \dd w
=   A(z) - a_{k-1}B(z),
\]
say. Next, we rewrite $A(z)$ as $A(z)=A_1(z)-A_2(z)$ where $A_1$ denotes the integral obtained by
extending the upper limit of integration in $A$ to $\ln z$, and $A_2$ denotes the integral over
$[\ln(z - \frac14),\ln z]$.
We note that for $z\geq e_k$ one has $\ln(z - \frac14) \geq e_{k-2}$, so that every
$\ln_j$ is positive and increasing over both integration ranges. Then
\[
A_2(z)
\leq \big(\ln z-\ln(z-1/4)\big)\prod_{j=1}^{k-1}\ln_{j}(\ln z)
=    -\ln(1-1/(4z))\prod_{j=2}^{k}\ln_{j} z.
\]
As for $A_1$, integration by parts (recalling that $\ln_{k-1}(e_{k-2})=0$) yields
\begin{align*}
A_1(z)
&=    \int_{e_{k-2}}^{\ln z} \prod_{j=1}^{k-1}\ln_{j} w\, \dd w
=      w\prod_{j=1}^{k-1}\ln_{j} w \Big|_{e_{k-2}}^{\ln z}
     - \int_{e_{k-2}}^{\ln z} w\Big(\prod_{j=1}^{k-1}\ln_{j} w\Big)' \dd w              \\
&=     \prod_{j=1}^{k}\ln_{j} z
     - \int_{e_{k-2}}^{\ln z} \Big(\sum_{i=2}^{k}\prod_{j=i}^{k-1}\ln_{j} w\Big) \dd w
\geq   \prod_{j=1}^{k}\ln_{j} z
     - (\ln z - e_{k-2})\sum_{i=2}^{k}\prod_{j=i+1}^{k}\ln_{j} z.
\end{align*}
We notice that for $k\geq 3$ and $z\geq e_k$ we have the chain of elementary inequalities:
\[
e_{k-2}\sum_{i=2}^{k}\prod_{j=i+1}^{k}\ln_{j} z
 \geq \prod_{j=3}^{k}\ln_{j} z
 =    \frac{1}{\ln_2 z}\prod_{j=2}^{k}\ln_{j} z
 \geq -\ln(1-1/(4z))\prod_{j=2}^k \ln_j z
 \geq A_2(z).
\]
Putting everything together and collecting a factor of $\prod_{j=1}^{k-1}\ln_{j} z$, we get
\[
A(z)\geq \prod_{j=1}^{k}\ln_{j} z - \bigg(\sum_{i=2}^{k}\ln_k z\prod_{j=2}^{i}\frac{1}{\ln_{j} z}\bigg) \prod_{j=1}^{k-1}\ln_{j} z.
\]
We are assuming $z\geq e_{k}$, and each quotient $\ln_k z /(\ln_2 z\cdots \ln_i z)$ decreases in this
range, hence it can be bounded by its value at $e_k$, producing the bound
\begin{align*}
\sum_{i=2}^{k}\ln_k z\prod_{j=2}^{i}\frac{1}{\ln_{j} z}
&\leq \sum_{i=0}^{k-2}\prod_{j=i}^{k-2}\frac{1}{e_{j}}
 =    \frac{1}{e_{k-2}}\Big(1 + \frac{1}{e_{k-3}} + \frac{1}{e_{k-3}e_{k-4}} + \cdots + \frac{1}{e_{k-3}e_{k-4}\cdots e_0}\Big)\\
&\leq \frac{1}{e_{k-2}}\Big(1 + \frac{k-2}{e_{k-3}}\Big).
\end{align*}
Hence,
\begin{equation}\label{eq:10C}
A(z)
\geq \prod_{j=1}^{k}\ln_{j} z
     - \Big(\frac{1}{e_{k-2}} + \frac{k-2}{e_{k-2}e_{k-3}} \Big) \prod_{j=1}^{k-1}\ln_{j} z.
\end{equation}

Next, we bound $B$. We extend its domain up to $\ln z$ by positivity of the integrand, getting
\begin{equation}\label{eq:11C}
B(z)
\leq \int_{e_{k-2}}^{\ln z} \prod_{j=1}^{k-2}\ln_{j} w\, \dd w
\leq \ln z \prod_{j=1}^{k-2}\ln_{j}(\ln z)
=    \prod_{j=1}^{k-1}\ln_{j} z.
\end{equation}
By~\eqref{eq:10C} and~\eqref{eq:11C} we get the claim.
\end{proof}
\begin{lemma}\label{lem:9C}
Let $k\in\{1,2,3\}$ and $z \geq e_k$. Then
\[
T_k(z) \geq \prod_{j=1}^k \ln_j z - a_k\prod_{j=1}^{k-1} \ln_j z
\]
with %
$a_1 = 0$,
$a_2 = 1$
and
$a_3 = 1.28$.
\end{lemma}
\begin{proof}
The claim for $T_1$ follows from its definition. Moreover, for $z > h_2$ we have
\begin{align*}
T_2(z)
 =    \int_1^{\ln(z - 1/4)} \hspace{-0.5cm} T_1(u) \dd u
 =    \int_1^{\ln(z - 1/4)} \hspace{-0.5cm} \ln u\, \dd u
 =    (u\ln u - u)\Big|_{1}^{\ln(z - 1/4)}
\geq  \ln z \ln_2 z - \ln z,
\end{align*}
as one can see via elementary computations.
inserting the lower bound we just proved for $T_2$, we get
\begin{align*}
T_3(z)
\geq \int_{h_2}^{\ln(z - 1/4)} \hspace{-0.5cm} \ln u(\ln_2 u - 1)\dd u
=    \big(u\ln u\ln_2 u - u\ln u - u\ln_2 u + \Li( u)\big)\Big|_{h_2}^{\ln(z - 1/4)},
\end{align*}
for $z>h_3$ and where $\Li(u):= \int_2^u \frac{\dd s}{\ln s}$.
With elementary computations one proves from this bound that
\[
T_3(z) \geq  \ln z \ln_2 z\ln_3 z - 1.28\ln z\ln_2 z
\]
when $z\geq h_3$.
\end{proof}
\begin{proof}[Proof of Theorem~\ref{thm:lower-bound}]
The first claim immediately follows from Lemma~\ref{lem:1C} and the lower bound~\eqref{eq:2C} in
Lemma~\ref{lem:5C}.

The second claim follows from Lemma~\ref{lem:1C}, the definition of $G$ in~\eqref{eq:5C}, the case $k=1$
of the lower bound~\eqref{eq:7C} and the definition of $T_1$. In fact, this
argument gives the claimed lower-bound already for $\ln_2 N \geq 1$ but this bound improves on the previous
one for $\ln_3 N \geq 1$ only.

Finally, let $k\geq 4$. Iterating the conclusion in Lemma~\ref{lem:8C} we have~\eqref{eq:9C} with
\[
a_k \leq a_3 + \sum_{j=4}^{\infty} \Big(\frac{1}{e_{j-2}} + \frac{j-2}{e_{j-2}e_{j-3}}\Big)
    \leq a_3 + 0.12.
\]
By Lemma~\ref{lem:9C} we can pick $a_3=1.28$, hence
\begin{equation}\label{eq:12C}
T_k(z) \geq \Big(1  - \frac{1.4}{\ln_k z}\Big)\prod_{j=1}^k \ln_j z
\end{equation}
for $z\geq e_k$, for every $k\geq 4$. By Lemma~\ref{lem:9C} the same holds for $k=2$ and $3$.
By~\eqref{eq:12C} (with a shift $k\to k-2$),~\eqref{eq:7C},~\eqref{eq:5C}, and Lemma~\ref{lem:1C} we
deduce that
\[
\frac{\ln(\card{\EE_N})}{\ln 2}
\geq 2\Big(1 - \frac{1.4}{\ln_k N}\Big)\frac{N}{\ln N}\prod_{j=3}^{k} \ln_j N
\]
for every $k\geq 4$.
\end{proof}

\section{Computing \texorpdfstring{$\card{\EE_N}$}{EEN}}\label{sec:computation}

In this section we explain how we computed $\card{\EE_N}$ for $N = 1,\dots,154$.
For every finite set of positive integers $S$, let
\begin{equation*}
    \EE(S) := \bigg\{\sum_{n \in S} \frac{t_n}{n}\colon  t_n \in \{0,1\}\ \forall n\in S\bigg\}
\end{equation*}
and
\begin{equation*}
    \UU(S) := \Big\{N \geq 1\colon \sum_{n \in S} \frac{w_n}{n}
    \neq \frac{1}{N} \text{ for all } (w_n)_{n \in S} \in \{-1,0,+1\}^{\card{S}}\Big\} .
\end{equation*}
Then it is straightforward to generalize Lemma~\ref{lem:0C} and Lemma~\ref{lem:1C} to the following results.

\begin{lemma}\label{lem:0Z}
    Let $s_0$ be a positive integer and let $S$ be a finite set of positive integers not containing $s_0$.
    Then $s_0 \in \UU(S)$ if and only if $\card{\EE(S \cup \{s_0\})} = 2\card{\EE(S)}$.
\end{lemma}

\begin{lemma}\label{lem:1Z}
    Let $S_1$ and $S_2$ be disjoint finite set of positive integers.
    Then
    \begin{equation*}
        \card{\EE(S_1 \cup S_2)} \geq 2^{\card{\UU(S_1)}} \EE(S_2) .
    \end{equation*}
\end{lemma}

Let us see how to compute $\card{\EE_{154}}=\card{\EE(\af)}$ for $\af := \{1, \dots, 154\}$.
The computation of the values $\EE_N$ for $N < 154$ is similar.

Let $p^k$ be any prime power which is larger than $154/2=77$. Then the equality
\[
\sum_{n\in\af\backslash\{p^k\}} \frac{w_n}{n}
= \frac{1}{p^k}
\]
with $w_n\in\{-1,0,1\}$ is impossible.
Thus, we can remove $p^k$ from $\af$, getting a shorter set $\af'$,
and by Lemma~\ref{lem:0Z} the cardinality of the set $\EE(\af)$ is two times the
cardinality of $\EE(\af')$.
Repeating this step several times we remove from the original set the numbers $128$, $81$, $125$, $121$,
$79$, $83$, $89$, $97$, $101$, $103$, $107$, $109$, $113$, $127$, $131$, $137$, $139$, $149$, $151$
getting a new set $\af'$.

Now consider a number of the form $2p^k$ and suppose that $p^k>154/3$. Then $p^k$ and $2p^k$ are the
unique elements in $\af^\prime$ which are divisible by $p^k$. Suppose moreover that $p\neq 3$, which
means that it is compatible with $m=2$, so that the equality
\[
\sum_{n\in\af'\backslash\{p^k,2p^k\}} \frac{w_n}{n} = \frac{1}{p^k} \Big(w + \frac{1}{2}\Big)
\]
with $w_n$ and $w\in\{-1,0,1\}$ is impossible. As a consequence, these numbers $2p^k$ can also be removed
from $\af'$, because their presence simply doubles the size of the set of Egyptian fractions. Repeating
this step we can remove from $\af'$ the numbers $128$ (which has already been removed), $106$, $118$,
$122$, $134$, $142$, $146$. Now that these numbers of the form $2p^k$ have been removed, we can also
remove the corresponding $p^k$ (applying the argument we have used as first step).
Thus, we also remove $64$, $53$, $59$, $61$, $67$, $71$, $73$, getting a new set $\af''$.

Now, consider a number of the form $3p^k$ and suppose that $p^k>154/4$. Then $p^k$, $2p^k$ and $3p^k$
are the unique elements in $\af'$ which are divisible by $p^k$. Suppose moreover that $p>g_3=11$, so that
it is compatible with $m=3$, and the equality
\[
\sum_{n\in\af''\backslash\{p^k,2p^k,3p^k\}} \frac{w_n}{n}
= \frac{1}{p^k} \Big(w + \frac{w'}{2} + \frac{1}{3}\Big)
\]
with $w_n$ and $w$, $w'\in\{-1,0,1\}$ is impossible. As above, these $3p^k$ can be removed from $\af''$.
In this way we remove from $\af''$ the numbers $123$, $129$, $141$. Now that these numbers of the form
$3p^k$ have been removed, we can also remove the corresponding $2p^k$ (applying the argument we have
used as second step), and the corresponding $p^k$ (applying the argument we have used as first step).
Thus, we also remove $82$, $86$, $94$, $41$, $43$, $47$, getting a new set $\af'''$.

Finally, consider a number of the form $4p^k$, where $p^k > 154/5$ and $p>g_4=25$. As above, we can remove
from $\af'''$ the numbers $124$ and $148$, and then, in cascade, the numbers $93$, $111$, $62$, $74$,
$31$, $37$, getting finally the set $\af^{iv}$.
At this point, we have $|\EE_{154}|=2^{49}|\EF(\af^{iv})|$.
The set $\af^{iv}$ contains only 105 elements,
but it is still too large to directly compute the corresponding set of Egyptian fractions.
The next remark is slightly more involved.

We subdivide $\af^{iv}$ into two disjoint subsets $\af_0$ and
$\af_1:=\{29n\colon 1\leq n\leq 5\}$. Now $\af_0$ has only 100 elements, and it is computationally
feasible to directly compute $E_0:=\EF(\af_0)$. We then have
\[
E':=\EF(\af^{iv})=\bigcup_{a\in\EF(\af_1)}(E_0+a).
\]
It is obvious that $\EF(\af_1)=\frac{1}{29}\EE_5$. Let $\ell:=\lcm(\af_0)$. By construction $\ell$ is
prime to $29$. We obviously have $|E'|=|29\ell E'|$ and
\[
29\ell E'=\bigcup_{a\in 29\ell\EF(\af_1)}(29\ell E_0+a).
\]
Now, if $a$ and $b\in 29\ell\EF(\af_1)$ are such that $a$ and $b$ have different classes modulo $29$, then
$29\ell E_0+a$ and $29\ell E_0+b$ are disjoint. The elements of $29\ell\EF(\af_1)=\ell\EE_5$ reduce modulo
$29$ in $24$ classes,
Among those $8$ classes, the difference between the two representatives is $d_1:=29\cdot 37479602160$ for
$2$ of them and $d_2:=29\cdot 56219403240$ for the remaining $6$. Therefore,
\[
|E'|=|\ell E'|=16|E_0| + 2|E_0\cup(E_0+d_1/29)| + 6|E_0\cup(E_0+d_2/29)|.
\]
What we have gained is that $\lcm(\af^{iv})=29\ell$ and, since the elements of $E_0$ (or $E'$, if we had
computed it) are represented by bitfields, we divide by $29$ the requested memory and time for the
computation.
To furthermore halve the size of the sets we need to consider, we observe that any $\EF(S)$ has a center
of symmetry (by induction because if $E\subseteq\R$ is symmetric, then for any $x\in\R$, $E\cup(E+x)$ is
symmetric, with its center moved by $x/2$ with respect to the center of $E$).

Keeping note of the intermediate cardinalities of the sets
\[
\Big\{\sum_{n\in\af_0,\ n\leq N} \frac{t_n}{n}\colon t_n\in\{0,1\},\ \forall n\in\af_0\Big\},
\]
and taking in consideration the cancelled elements, we can compute the cardinalities of $\EE_{N}$ for all
$N\leq 154$.
Notice that this works only because the multiples of $29$ can be cancelled from the generating set of
$\EE_N$, for all $N\leq 5\cdot 29-1$.

The full computation needed approximatively two hours and 300GB RAM memory on the PlaFRIM platform in
Bordeaux.
%


\newpage

\begin{table}[ht]
\centering
\begin{tabular}{RR|RR|RR}
  \toprule
   N & \card{\EE_N} &  N &        \card{\EE_N} &   N &                 \card{\EE_N} \\
  \midrule
   0 &            1 &  52 &          570733363200 & 104 &         434404550383671181312 \\
   1 &            2 &  53 &         1141466726400 & 105 &         435821848359665139712 \\
   2 &            4 &  54 &         1721081528320 & 106 &         871643696719330279424 \\
   3 &            8 &  55 &         1751601381376 & 107 &        1743287393438660558848 \\
   4 &           16 &  56 &         1767017021440 & 108 &        1754513627060579074048 \\
   5 &           32 &  57 &         3534034042880 & 109 &        3509027254121158148096 \\
   6 &           52 &  58 &         7068068085760 & 110 &        3522492005456298377216 \\
   7 &          104 &  59 &        14136136171520 & 111 &        7044984010912596754432 \\
   8 &          208 &  60 &        14245758500864 & 112 &        7068497418916307402752 \\
   9 &          416 &  61 &        28491517001728 & 113 &       14136994837832614805504 \\
  10 &          832 &  62 &        56983034003456 & 114 &       16899242066544045326336 \\
  11 &         1664 &  63 &        57494604873728 & 115 &       22272240164078654324736 \\
  12 &         1856 &  64 &       114989209747456 & 116 &       44544480328157308649472 \\
  13 &         3712 &  65 &       137824242237440 & 117 &       44696007986571758272512 \\
  14 &         7424 &  66 &       139033409748992 & 118 &       89392015973143516545024 \\
  15 &         9664 &  67 &       278066819497984 & 119 &       90161265693495021010944 \\
  16 &        19328 &  68 &       522131016253440 & 120 &       90370501722719863701504 \\
  17 &        38656 &  69 &      1044262032506880 & 121 &      180741003445439727403008 \\
  18 &        59264 &  70 &      1051387483914240 & 122 &      361482006890879454806016 \\
  19 &       118528 &  71 &      2102774967828480 & 123 &      722964013781758909612032 \\
  20 &       126976 &  72 &      2116809947873280 & 124 &     1445928027563517819224064 \\
  21 &       224128 &  73 &      4233619895746560 & 125 &     2891856055127035638448128 \\
  22 &       448256 &  74 &      8467239791493120 & 126 &     2899478785052218761412608 \\
  23 &       896512 &  75 &     10638462277386240 & 127 &     5798957570104437522825216 \\
  24 &       936832 &  76 &     17372520791408640 & 128 &    11597915140208875045650432 \\
  25 &      1873664 &  77 &     17522758873251840 & 129 &    23195830280417750091300864 \\
  26 &      3747328 &  78 &     17647454272880640 & 130 &    23285154296843172833132544 \\
  27 &      7494656 &  79 &     35294908545761280 & 131 &    46570308593686345666265088 \\
  28 &      7771136 &  80 &     35499851152097280 & 132 &    46685922910553749195849728 \\
  29 &     15542272 &  81 &     70999702304194560 & 133 &    46892844848672951268016128 \\
  30 &     15886336 &  82 &    141999404608389120 & 134 &    93785689697345902536032256 \\
  31 &     31772672 &  83 &    283998809216778240 & 135 &    94091474339233167820455936 \\
  32 &     63545344 &  84 &    285080778587504640 & 136 &    94380290287482179111878656 \\
  33 &    112064512 &  85 &    326182987337039872 & 137 &   188760580574964358223757312 \\
  34 &    224129024 &  86 &    652365974674079744 & 138 &   207845467988940343115513856 \\
  35 &    231010304 &  87 &   1304731949348159488 & 139 &   415690935977880686231027712 \\
  36 &    237031424 &  88 &   1312124045747027968 & 140 &   416574753297226313332948992 \\
  37 &    474062848 &  89 &   2624248091494055936 & 141 &   833149506594452626665897984 \\
  38 &    948125696 &  90 &   2637135095231676416 & 142 &  1666299013188905253331795968 \\
  39 &   1896251392 &  91 &   2653366206139990016 & 143 &  1670768625140679902219993088 \\
  40 &   1928593408 &  92 &   5306732412279980032 & 144 &  1674506556739064055000465408 \\
  41 &   3857186816 &  93 &  10613464824559960064 & 145 &  2520048783079754452174897152 \\
  42 &   3925999616 &  94 &  21226929649119920128 & 146 &  5040097566159508904349794304 \\
  43 &   7851999232 &  95 &  26280284845565280256 & 147 &  8826629661276147104436191232 \\
  44 &  12445024256 &  96 &  26477620983450566656 & 148 & 17653259322552294208872382464 \\
  45 &  12606504960 &  97 &  52955241966901133312 & 149 & 35306518645104588417744764928 \\
  46 &  25213009920 &  98 & 105910483933802266624 & 150 & 35407794518497650679945887744 \\
  47 &  50426019840 &  99 & 106363570685376200704 & 151 & 70815589036995301359891775488 \\
  48 &  51334348800 & 100 & 107289734959184478208 & 152 & 71018819275510828026883473408 \\
  49 & 102668697600 & 101 & 214579469918368956416 & 153 & 71201999904617906814717001728 \\
  50 & 205337395200 & 102 & 215966239954017714176 & 154 & 71356425097301949080433328128 \\
  51 & 410674790400 & 103 & 431932479908035428352                                       \\
  \bottomrule
\end{tabular}
\medskip
\caption{The first values of $\card{\EE_N}$.}
\label{tab:2}
\end{table}

\clearpage
\end{document}